\documentclass[11pt]{amsart}

\usepackage{amsmath, amssymb, amsfonts, amsbsy, latexsym}
\usepackage{multirow}
\usepackage{epsfig}
\usepackage{graphicx}
\def\argmin{\mathop{\rm argmin}}
\def\bfu{{\bf u}}
\def\bfv{{\bf v}}
\def\bff{{\bf f}}
\def\calA{{\mathcal A}}
\def\calAt{{\mathbb A}}
\def\calDt{{\mathbb D}}
\def\calIt{{\mathbb I}}

\newcommand{\R}{\mathbb{R}}
\newcommand{\ZZ}{\mathbb{Z}}
\newcommand{\II}{\mathbb{I}}

\newcommand{\CC}{\mathbb{C}}

\newcommand{\bfF}{{\bf F}}
\newcommand{\bPsi}{{\bf \Psi}}

\def\argmin{\mathop{\rm argmin}}

\def\AA{{\mathrm A}}
\def\calMt{{\mathbb M}}
\def\calIt{{\mathbb I}}
\def\calWt{\mathbb{W}}

\def\CC{{\Lambda}} 

\newtheorem{theorem}{Theorem}[section]
\newtheorem{corollary}[theorem]{Corollary}
\newtheorem{lemma}[theorem]{Lemma}

\newtheorem{remark}[theorem]{Remark}

\title[Comparison of URA methods for fractional diffusion problems]{Comparison analysis on two numerical methods for fractional diffusion problems based on rational approximations of $t^{\gamma}, \ 0  \le t \le 1$}
\author[S. Harizanov, R. Lazarov, S. Margenov, P. Marinov, J. Pasciak]
{Stanislav Harizanov \and Raytcho Lazarov \and Pencho Marinov \and Svetozar Margenov \and Joseph Pasciak}
\address{Institute of Information and Communication Technologies, Bulgarian Academy of 
Sciences, Acad. G. Bonchev, bl. 25A, 1113 Sofia, Bulgaria
(sharizanov@parallel.bas.bg)}
\address{Deptartment of Mathematics, Texas A\&M University, 
College Station, TX 77843-3368, USA and Institute of Mathematics and Informatics,
Bulgarian Academy of Sciences, Acad. G. Bonchev, bl. 8, 1113 Sofia, Bulgaria (lazarov@math.tamu.edu)}
\address{Institute of Information and Communication Technologies, Bulgarian Academy of 
Sciences, Acad. G. Bonchev, bl. 25A, 1113 Sofia, Bulgaria (pencho@parallel.bas.bg)}
\address{Institute of Information and Communication Technologies, Bulgarian Academy of 
Sciences, Acad. G. Bonchev, bl. 25A, 1113 Sofia, Bulgaria (margenov@parallel.bas.bg)}
\address{Deptartment of Mathematics, Texas A\&M University, 
College Station, TX 77843-3368, USA (pasciak@math.tamu.edu)}

\begin{document}

\begin{abstract}
We discuss, study, and compare experimentally three methods for solving the system of algebraic equations
$\calAt^\alpha \bfu=\bff$, $0< \alpha <1$, where 
$\calAt$ is a symmetric and positive definite matrix obtained 
from finite difference or finite element approximations of second 
order elliptic problems in $\R^d$, $d=1,2,3$. The first method, 
introduced in \cite{HLMMV18}, based on the best
uniform rational approximation (BURA) $r_\alpha(t)$ of 
$t^{1-\alpha}$ for $0 \le t \le 1$, is used to get the rational
approximation of  $t^{-\alpha}$ in the form $t^{-1}r_\alpha(t)$.
Here we develop another method, denoted by R-BURA, that is 
based on the best rational approximation $r_{1-\alpha}(t)$ of
$t^\alpha$ on the interval $[0,1]$  and approximates $t^{-\alpha}$ via $r^{-1}_{1-\alpha}(t)$. 
The third method, introduced and studied by Bonito and Pasciak in \cite{BP15}, 
is based on an exponentially convergent quadrature scheme for the Dundord-Taylor integral representation 
of the fractional powers of elliptic operators.
All three methods reduce the solution of the system $\calAt^\alpha \bfu=\bff$ to solving
a number of equations of the type $(\calAt +c \II)\bfu= \bff$, $c \ge 0$. 
Comprehensive numerical experiments on
model problems with $\calAt$ obtained by approximation of elliptic equations
in one and two spatial dimensions are used to compare the efficiency of these three
 algorithms depending on the fractional power $\alpha$. 
The presented results prove the concept of the new R-BURA method,  which performs well for $\alpha$ close to $1$
in contrast to BURA, which performs well for $\alpha $ close to $0$. As a result, we show theoretically and experimentally,
that they have mutually complementary advantages.

\end{abstract}

\maketitle

\section{Introduction}\label{section1}
\subsection{Algebraic problems under consideration } 
\label{sec:problem}

Now let $\R^N$, $N$ positive integer, be an $N$-dimensional vector space 
with the standard $\ell_2$-inner product, $\bfu^T \bfv$,  for any real vectors $\bfu, \bfv \in \R^N$,
and let
$\calAt  \in \R^{N \times N}$ be a symmetric and positive definite matrix 
with eigenvalues and eigenvectors $\{ (\lambda_i, \bPsi_i) \}_{i=1}^N $.
We assume that the eigenvectors are orthonormal, that is
$\bPsi_i^T \bPsi_j = \delta_{ij}$,  and $0 <\lambda_1 \le \lambda_2 \le \dots \lambda_N$.

For $0< \alpha <1$  and given $\bff \in \R^N$ we consider the following algebraic problem: 
\begin{equation}\label{eq:fal}
\mbox{find } ~~ \bfu \in \R^N ~~ \mbox{such that } \quad
{\calAt}^\alpha \bfu =   \bff
\end{equation}
where the fractional power $\calAt^\alpha$ is defined through eigenvalues and eigenvectors of $\calAt$
$$
\calAt^\alpha = \calWt \calDt^\alpha \calWt^T, \ \  \mbox{ where } \ \  \calAt=\calWt \calDt \calWt^T.
$$
Here  
$\calWt, \calDt \in \R^{N \times N}$ are
defined as $\calWt=[\bPsi_1^T, \bPsi_2^T, ..., \bPsi_N^T]$ and
$\calDt=diag(\lambda_1, \dots, \lambda_N)$.
Then $ \calAt^{-\alpha}= \calWt \calDt^{-\alpha} \calWt^T$ and the solution of 
$\calAt^\alpha \bfu =\bff$ can be expressed as
\begin{equation}\label{eqn:exacalg}
{\bfu} = \calAt^{-\alpha} \bff =\calWt \calDt^{-\alpha} \calWt^T \bff.
\end{equation}
For any $\beta \in \R$ we have  
$
\| \bfu \|_{\calAt^{\beta+\alpha}}= \| \bff \|_{\calAt^{\beta-\alpha}} $, where
$
\| \bfu\|^2_{\calAt^\gamma} = \bfu^T \calAt^{\gamma} \bfu, \  \gamma \in \R .
$

\newpage
\subsection{Examples of SPD matrices under consideration} 
\subsubsection{Example 1.}
The first example of such a matrix is $\calAt \in \R^{N \times N}$, $N=n^2$, that has the
following block stricture (here $ \calAt_{i,i} \in \R^{n \times n}$, $i=1, \cdots, n $ 
and $\calIt_n$ is the identity matrix in $\R^n$) 
\begin{equation}\label{FD-matrix} 
\calAt^{}=
(n+1)^2 \left[\begin{array}{ccccc} \calAt_{1,1} & -\calIt_n &&&\\ -\calIt_n & \calAt_{2,2} & -\calIt_n 
&&\\\cdots &\cdots &\cdots &\cdots &\cdots\\ & -\calIt_n & \calAt_{i,i} & -\calIt_n &\\\cdots &\cdots &\cdots &\cdots 
&\cdots\\&&& -\calIt_n & \calAt_{n,n}\end{array}\right],\quad \calAt_{i,i}=\left[\begin{array}{cccc} 
4&-1&&\\-1&4&-1&\\\cdots&\cdots&\cdots&\cdots\\&-1&4&-1\\&&-1&4\end{array}\right].
\end{equation}

This matrix is generated by the finite difference approximation of the 
following boundary value problem 
\begin{equation}\label{2D Laplace}
\begin{split}
-\Delta u = f, \ \  x \in \Omega=(0,1) \times (0,1), \ \ 
u =0, \ \ x \in \partial \Omega
\end{split}
\end{equation}
on a uniform mesh with mesh-size $h=1/(n+1)$.

For a given $\bff \in \R^{N}$  the algebraic problem 
\begin{equation}
\label{FD4Laplace}
\mbox{ find  } \bfu \in \R^{N} \ \ \mbox{ that satisfies } \ \ 
{\calAt}^\alpha \bfu =   \bff 
\end{equation}
is an approximation to the boundary value problem
\begin{equation}\label{FracLaplace}
(-\Delta)^\alpha u = f, \ \  x \in \Omega=(0,1)\times (0,1), \ \ 
u =0, \ \ x \in \partial \Omega
\end{equation}
with $\bff$ being a projection of $f(x)$ onto
the space of mesh functions.

We are not aware of rigorous analysis of the approximation properties of 
problem \eqref{FD4Laplace}. 
However, this linear system approximates \eqref{2D Laplace}
due to its relation of the finite element approximation on a triangular mesh, discussed below.

\subsubsection{Example 2.}
We partition $\Omega=(0,1) \times(0,1)$ into squares of size $h=1/(n+1)$.
Let ${\mathcal T}_h$ be  obtained by subdividing each square into two 
triangles by connecting the upper left 
corner with the lower right corner. On this triangulation we introduce the 
space $V_h \subset H^1_0(\Omega)$ of continuous piece-wise linear function. The 
finite element approximation of \eqref{2D Laplace}  is:  find $u_h \in V_h$ such that
\begin{equation}
\label{Lapalce_weak}
a(u_h,v) := \int_\Omega   \nabla u_h(x) \cdot  \nabla v(x) dx = (f,v)  := (\pi_h f, v)\ \ \forall v \in V_h.
\end{equation}
Here $(\cdot, \cdot) $ is the standard $L_2$-inner product on $V_h$.
We define the operator $\AA: V_h \to V_h$ by 
$(\AA u,v)= (z,v )$ for all $v \in V_h$.  Then
$\AA y=z \in V_h$ should have representation through the nodal basis $\psi_k$: $z=\sum c_k \psi_k$
and then the operator $\AA$ is expressed through the global ``stiffness" matrix 
$ \{ {\mathbb A} \}_{i,k}= a(\psi_i,\psi_k) $ 
and the global ``mass" matrix ${\mathbb M} =\{ (\psi_i,\psi_k)\}_{i,k} $
 via the relation $ \AA= {\calMt}^{-1} {\calAt} $.

The matrix ${\calMt}$ is not diagonal and has similar sparsity pattern as the 
the ``stiffness" matrix $ {\calAt}$.  The algebraic problem
\begin{equation}
\label{FEM}
\AA^\alpha u_h = \pi_h f
\end{equation}
is stable and 
$
u_h(x) = \sum_{i=1}^{N} u_i \psi_i(x), \  \bfu=\{u_1,\cdots, u_{N}\}^T
$
approximates the solution $u(x)$ of \eqref{FracLaplace}, see, \cite{BP15}.

In order to get the matrix $\calAt$,
as need in \eqref{FD4Laplace},  (instead of  $ {\AA}$) we can apply the following approach.
First, we introduce the ``lumped" mass inner product in $V_h$, \cite[pp.239--242]{Thomee2006}.
Namely, for $z,v \in V_h$ we define
$$
(z,v)_h = \frac13 \sum_{\tau \in {\mathcal T}_h } \sum_{i=1}^3 |\tau| \ z(P_i) v(P_i),
\ \ \mbox{and lumped ``mass" matrix} \ \ {\calMt} =\{ (\psi_i,\psi_k)_h\}_{i,k},
$$
where $P_1, P_2, P_3$ are the vertexes of the triangle $\tau$ and $|\tau|$ is its area.
Note that the lumped mass  matrix ${\calMt}_h$ is diagonal!
Moreover, since the mesh is square, all diagonal elements of $ {\calMt}^{-1}_h$ are equal to 
$ h^{-2}=(n+1)^{2}$ and  $  {\AA}: V_h \to V_h$ is defined by
$$
 ({\AA } u_h, v)_h = a(u_h,v) \ \ 
 \mbox{ gives } \ \ \AA = {\calMt}^{-1} {\calAt}.
$$ 
Since on a uniform mesh $(\cdot, \cdot)_h $ is a good approximation 
of $(\cdot,\cdot)$ then one concludes that \eqref{FD4Laplace} approximates the 
problem \eqref{FracLaplace}.

\subsubsection{Example 3.}

Similar considerations could be also used in solving elliptic problem with non-constant 
coefficient in the reaction term generated by the weak form
\begin{equation}\label{eqn:weak}
A(u_h,v) := 
a(u_h, v) +(qu_h , v)_h =(f, v) \ \ \forall v \in V_h,
\end{equation}
with  $q=q(x) \ge 0$. The bilinear form is symmetric and coercive on $V_h$ 
and the corresponding algebraic problem will have the same properties as the one
involving Laplace operator. Then the matrix of the corresponding linear system is the sum
$
\calAt^{} + \calDt, 
$
where 
$\calDt$ is a diagonal matrix in  $\R^{N}$ with entries the values of $q(x)$ at the mesh points.

\begin{remark}\label{Neumann}
One can generate matrices with similar structure while solving elliptic problems with Neumann 
or Robin boundary conditions. 
\end{remark}

\subsection{The concept of the best uniform rational approximation (BURA)}\label{ss:rat}

We shall use the following notation  for a class of rational functions:  
$$
\mathcal R(k,m)= \{ r(t): ~ r(t)=P_k(t)/Q_m(t), \ \ \mbox{where} \ P_k \in {\mathcal P}_k, \ 
\mbox{and  } \ Q_m \in {\mathcal P}_m, \ \mbox{monic} \}
$$
with  ${\mathcal P}_j$ set of algebraic polynomials of degree $j$.
The best uniform approximation $r_\alpha(t) \in \mathcal R(k,m)$ of $t^{1-\alpha}$ on $[0,1]$ 
(called further $(k,m)$-BURA), and its 
approximation error $E_{\alpha,k,m}$ are defined as follows:
$$
r_\alpha(t) :=\argmin_{ r \in \mathcal R(k,m)} \left \| t^{1-\alpha} - r(t)  \right \|_{L^\infty(0,1)},\ \ 
\varepsilon(t) = r_{1-\alpha}(t)- t^\alpha, \ \ 
E_{\alpha,k,m}:= \| \varepsilon (t) \|_{L^\infty(0,1)}.
$$
For $k=m$ the existence and uniqueness of $r_\alpha(t)$ has been established long time ago 
\cite[Chapters 9.1 and 9.2]{Meinardus67}. Moreover, it is known that both the numerator and the denominator 
of the minimizer are of exact degree $k$ and the error function $\varepsilon(t)$ possesses $2k+2$ 
extreme points in $[0,1]$, including the endpoints of the interval.

\subsection{Methods for solving equations involving functions of 
matrices}\label{methods}

The formula \eqref{eqn:exacalg}  could be used in practical 
computations if the eigenvectors and eigenvalues are explicitly known and  
Fast Fourier Transform is applicable to perform the matrix vector 
multiplication with $\calWt$, thus leading to almost optimal computational 
complexity, $O(N \log N)$. However, this approach is limited to 
separable problems with constant coefficients in simple domains and boundary conditions.

This work is related also to the more difficult problem of stable 
computations of the matrix square root and other functions of matrices, 
see, e.g. the earlier papers \cite{druskin1998extended,Higham1997,Kenney1991}, 
as well as, \cite{SFilip18} for some more recent related results.  
However, in this paper we do not deal with evaluation 
of $\calAt^{\alpha}$, instead we discuss efficient methods
for solving the algebraic system $\calAt^{\alpha} \bfu =\bff$, 
where $\calAt$ is an SPD matrix generated by 
approximation of second order elliptic operators.

Our research is also connected with the work done in  
\cite{ilic2005numerical,ilic2009numerical}, where numerical approximation 
of a fractional-in-space diffusion equation is considered.
In \cite{ilic2009numerical}, the proposed solver relies on Lanczos 
method. First, the adaptively preconditioned thick restart Lanczos 
procedure is applied to a system with $\calAt$.  The gathered spectral 
information is then used to solve the system with $\calAt^\alpha$. 
In \cite{druskin1998extended} an extended Krylov subspace method is 
proposed, originating by actions of the SPD matrix and its inverse. 
It is shown that for the same approximation quality, the variant of 
the extended subspaces requires about the square root of the dimension 
of the standard Krylov subspaces using only positive or negative matrix 
powers. A drawback of this method is the memory required to store 
the full dense matrix $\calWt$, and the substantial deterioration of the 
convergence rate for ill-conditioned matrices. The advantage of 
the approach discussed in this paper  is the robustness and almost 
optimal computational complexity.

\subsection{Our contributions}\label{contributions}
%
%
We investigate two approaches for approximate solving of   $\calAt^{-\alpha} \bf f$ that
are based on the best uniform rational
approximation (BURA) $r(t)$ of $t^\gamma$, $\gamma >0$, on $[0,1]$.
One subclass of such approximations is expressed through
diagonal Walsh table $P_k(t)/Q_k(t)$, i.e. $r \in {\mathcal R}(k,k)$,
see, e.g. \cite{stahl2003,varga1992some}.  
Another subclass is the upper diagonal $P_{k+1}(t)/Q_k(t)$,
i.e.,  $r \in  {\mathcal R}(k+1,k)$. 
The first analyzed method is introduced in
\cite{HLMMV18},  where the BURA $r_\alpha(t)$ of $t^{1-\alpha}$, 
$0 < t \le 1$, introduces a rational approximation of  
$t^{-\alpha}$ in the form $t^{-1}r_\alpha(t)$.
Here we develop a new method, denoted by R-BURA, where the best uniform rational
approximation $r_{1-\alpha} (t)$ of $t^\alpha$ is used to approximate $t^{-\alpha}$ by $1/ r_{1-\alpha}(t)$.
Both methods reduce solving \eqref{eq:fal} to a number of 
equations $(\calAt +c \calIt)\bfu= \bfF$, $c \ge 0$. 

Our comparative analysis includes also the method proposed in \cite{BP15} that is 
based on approximation of the integral representation of the solution of \eqref{FracLaplace}.
Then exponentially convergent quadrature formulae are applied to evaluate 
numerically the related integrals. In fact, this Q-method leads also to a rational approximation as well. 
The problem with checkerboard right hand side, introduced in \cite{BP15}, 
is used in the numerical tests of our comparative analysis.

The rest of the paper is organized as follows. In Section \ref{section2} we introduce
the basic properties of the studied solution methods and algorithms. 
The analysis includes error estimates of the
BURA, \cite{HLMMV18}, the new R-BURA, and Q-method of Bonito and Pasciak, \cite{BP15}.
Section \ref{sec:tests} contains numerical tests for fractional Laplace problems.
In the case of BURA and R-BURA solvers, the impact of scaling is
analyzed and experimentally confirmed. Among others, the numerical 
results complete the proof of concept of the new R-BURA approach,
illustrating its advantages in the case of larger $\alpha \in (1/2,1)$.

\section{Description of the numerical methods}\label{section2}
\subsection{The BURA method}
In this paper we consider  two BURA subclasses $(k,k)$ and $(k+1,k)$, 
introduced in Section~\ref{contributions}. Let $\CC:=\|\calAt\|_{\infty}=\max_{1\le i\le N}\sum_{j=1}^N|a_{ij}|$.
Following \cite{HLMMV18}, we obtain the rescaled SPD matrix 
$\calA:=\CC^{-1}\calAt$
with spectrum in  $(0,1]$. Then the original problem $\calAt^\alpha\bfu=\bff$ can be rewritten 
as $\calA^\alpha\bfu =\CC^{-\alpha}\bff$. Note that the eigenvalues of $\calA$ are 
$\mu_i:=\CC^{-1}\lambda_i$, $0< \mu_i \le 1$,
$i=1, \cdots, N$.

Let $r_\alpha(t)$ be  the  BURA of 
$t^{1-\alpha}$  in ${\mathcal R}(k,k)$ or ${\mathcal R}(k+1,k)$. Then 
%
\begin{equation}\label{eq:BURA}
 \bfu_r:=\CC^{-\alpha}\calA^{-1}r_\alpha(\calA)\bff,
\end{equation}
are called $(k,k)$-BURA and $(k+1,k)$-BURA approximation of $\bfu$, respectively.

Using the spectral decomposition of $\calAt$, we 
can derive the following estimation of the BURA error:
\begin{equation}\label{eq:BURA error analysis}
\frac{\|\bfu_r-\bfu \|_2}{\|\bff\|_2}=\CC^{-\alpha}\max_{\mu_i}\frac{\left|r_\alpha(\mu_i)-\mu_i^{1-\alpha}\right|}{\mu_i}
\le \frac{\CC^{1-\alpha}E_{\alpha,k,m}}{\lambda_1}.
\end{equation}
Then using \cite[Theorem 1]{Stahl93}  (about the behavior of $E_{\alpha,k.k}$ 
as $k\to\infty$) we get the following property of the $(k,k)$-BURA:
\begin{equation}\label{eq: BURA error}
\lim_{k \to \infty} e^{2\pi\sqrt{(1-\alpha)k} }  \|\bfu_r-\bfu \|_2
    =    \frac{4^{2-\alpha}\CC^{1-\alpha}}{\lambda_1}\sin(\pi\alpha)\|\bff\|_2. 
\end{equation}
 Since $E_{\alpha,k,k}\ge E_{\alpha,k+1,k}\ge E_{\alpha,k+1,k+1}$ by definition, \eqref{eq: BURA error}
 is valid for $(k+1,k)$-BURA, as well.
 
We restricted our experiments to $(k,k)$-BURA method.
The implementation uses the decomposition of the rational function 
 $t^{-1}r_\alpha(t)$ into sum of partial fractions so that
$$
\bfu_r=\sum_{j=0}^{k} c_j(\calA - d_j{\calIt})^{-1} \bff=\CC\sum_{j=0}^{k} c_j(\calAt - \CC d_j{\calIt})^{-1} \bff.
$$
Here $0=d_0>d_1\dots>d_k$ are the poles of $r_\alpha(t)$ plus the additional pole at zero, and $c_j>0$ for every $j$ 
(see \cite{SS93} for more details). Obviously, the approximation $\bfu_r$ is obtained by solving 
$k+1$ linear systems with nonnegative diagonal shifts of $\calAt$.

\subsection{The R-BURA method}

In this approach, we approximate $t^{-\alpha}$  
by $r^{-1}_{1-\alpha}(t)$  where $r_{1-\alpha}(t)$ is the best rational approximation of 
$t^\alpha$ in ${\mathcal R}(k,k)$ or ${\mathcal R}(k+1,k)$.
Then 
\begin{equation}\label{eq:R-BURA}
 \bfu_r:=\CC^{-\alpha}r^{-1}_{1-\alpha}(\calA)\bff
\end{equation}
are  called $(k,k)$-R-BURA and $(k+1,k)$-R-BURA approximation of $\bfu$, respectively.

For the analysis of the BURA-method we shall need the following properties of $r_\alpha(t)$:
\begin{lemma}\label{lemma:BURA monotonicity}
Let $\alpha\in(0,1)$ 
and $k$ be a positive integer. Then the best rational approximation (BURA)
 $r_\alpha(t) \in \mathcal R({k,k})$ of $t^{1-\alpha}$ in $[0,1]$ has the following properties:
\begin{itemize}
 \item[(a)] $r_\alpha(t)$ is strictly monotonically increasing concave function when $t\in[0,1]$;
\item[(b)] $r_\alpha(0)=E_{\alpha,k,k}$.
\end{itemize}
\end{lemma}
\begin{proof}
The second part follows directly from \cite[Lemma 2.1]{SS93}, where it is shown that $\eta_1=0$ is 
an extreme point for $t^{1-\alpha}-r_\alpha(t)$ with negative value. The same lemma states that all the 
$k$ zeros and $k$ poles (denoted by $d_j$) of $r_\alpha$ are real, pairwise different, non-positive, and interlacing. 
Then for the decomposition of $r_\alpha(t)$ into partial fractions
$$r_\alpha(t)=b^\ast_0+\sum_{j=1}^k\frac{c^\ast_j}{t-d_j}$$
we have   $c_j^\ast,d_j<0$, $j=1, \cdots, m$, for more details, see, e.g. \cite[Theorem 1]{harizanov2017positive}.
Hence,
\begin{equation*}
\begin{split}
r'_\alpha(t)=\sum_{j=1}^k\frac{-c^\ast_j}{(t-d_j)^2}>0,\qquad 
r''_\alpha(t)=\sum_{j=1}^k\frac{2c^\ast_j}{(t-d_j)^3}<0,\qquad\forall t\in[0,1].
\end{split}
\end{equation*}
The proof is completed.
\end{proof}

Applying Lemma~\ref{lemma:BURA monotonicity}, the $(k,k)$-R-BURA approximation error is 
estimated analogously to the BURA error: 
\begin{equation*}\label{eq:R-BURA error analysis 1}
\frac{\|\bfu_r-\bfu \|_2}{\|\bff\|_2}\le \CC^{-\alpha}\max_{\mu_i}\frac{\left|r_{1-\alpha}(\mu_i)-\mu_i^{\alpha}\right|}{\mu^\alpha_i r_{1-\alpha}(\mu_i)}\le\frac{\CC^{-\alpha}E_{1-\alpha,k,k}}{\mu^\alpha_1r_{1-\alpha}(\mu_1)}
\le \frac{E_{1-\alpha,k,k}}{\lambda^\alpha_1r_{1-\alpha}(\mu_1)}.
\end{equation*}
Note that, $r_{1-\alpha}(t)$
has no zeros inside the interval $t\in[0,1]$, therefore the denominator is strictly positive and the error bound is well-defined. On the other hand, $r_{1-\alpha}(0) = E_{1-\alpha,k,k}$ when 
$h\to 0$ then  $\mu_1 \to 0$ and the error deteriorates. 
The error function $\varepsilon(t)=r_{1-\alpha}(t)-t^\alpha$ has $2k+1$ roots $\{\xi_i\}_1^{2k+1}$ in $(0,1)$, 
due to the $2k+2$ extreme points, including $\{0,1\}$, see, e.g.,  \cite{SS93}. 
Since $\varepsilon(0)=E_{1-\alpha,k,k}>0$, we have 
\begin{equation}\label{eq:R-BURA interval}
r_{1-\alpha}(t)\ge t^\alpha,\qquad\forall t\in[0,\xi_1]\cup\bigcup\limits_{i=1}^k[\xi_{2i},\xi_{2i+1}]. 
\end{equation}
Therefore, whenever $\mu_1$ is a priori estimated (enough to have a good lower bound for $\lambda_1$), 
we can choose a proper $k$, such that 
\begin{equation}\label{eq:R-BURA error analysis}
\frac{\|\bfu_r-\bfu  \|_2}{\|\bff\|_2}\le \frac{E_{1-\alpha,k,k}}{\lambda^{\alpha}_1\mu^{\alpha}_1}= \frac{\CC^{\alpha}E_{1-\alpha,k,k}}{\lambda^{2\alpha}_1},\qquad \mu_1\in\bigcup\limits_{i=1}^k[\xi_{2i},\xi_{2i+1}].
\end{equation}
The case $\mu_1\in[0,\xi_1]$ is more subtle and needs special care. Using $r_{1-\alpha}(t)=t^\alpha+\varepsilon(t)$, with $0<\varepsilon(t)\le E_{1-\alpha,k,k}$, together with the fact that the function $g(\varepsilon):=\varepsilon/(\mu^\alpha_1+\varepsilon)$ is monotonically increasing for $\varepsilon\ge0$ we obtain
$$\frac{\left|r_{1-\alpha}(\mu_1)-\mu_1^{\alpha}\right|}{\mu^\alpha_1 r_{1-\alpha}(\mu_1)} \le \frac{E_{1-\alpha,k,k}}{\mu^\alpha_1(\mu^\alpha_1+E_{1-\alpha,k,k})}\le\frac{E_{1-\alpha,k,k}}{\mu^\alpha_1 E_{1-\alpha,k,k}}.$$
For every $\mu_i>\xi_1$ we have
$$\frac{\left|r_{1-\alpha}(\mu_i)-\mu_i^{\alpha}\right|}{\mu^\alpha_i r_{1-\alpha}(\mu_i)} \le \frac{E_{1-\alpha,k,k}}{\mu^\alpha_i r_{1-\alpha}(\xi_1)}\le \frac{E_{1-\alpha,k,k}}{\mu^\alpha_1 \xi^\alpha_1}.$$
Therefore, since $\xi^{\alpha}_1=r_{1-\alpha}(\xi_1)>r_{1-\alpha}(0)=E_{1-\alpha,k,k}$,
\begin{equation}\label{eq:R-BURA error analysis at 0}
\frac{\|\bfu_r-\bfu \|_2}{\|\bff\|_2}\le \frac{E_{1-\alpha,k,k}}{\lambda^{\alpha}_1\max(\xi^{\alpha}_1,E_{1-\alpha,k,k})}=
\lambda^{-\alpha}_1,\qquad \forall\mu_1\in[0,\xi_1].
\end{equation}

Typically $\lambda_1=\mathrm{O}(1)$ for 
all $h$ and $\mu_1\to 0$ as $h\to0$, thus unlike the BURA case \eqref{eq:BURA error analysis} the 
$(k,k)$-R-BURA relative error is uniformly bounded when $k$ is fixed and $h\to 0$.

The asymptotic behavior of the relative error \eqref{eq:R-BURA error analysis} is derived analogously to \eqref{eq: BURA error}:
\begin{equation}\label{eq: R-BURA error}
\lim_{k \to \infty} e^{2\pi\sqrt{\alpha k}} \|\bfu_r-\bfu \|_2 
 = \frac{4^{2-\alpha}\CC^{\alpha}}{\lambda^{2\alpha}_1}\sin(\pi\alpha)\|\bff\|_2. 
\end{equation}

In our experiments, we work with $r_{1-\alpha}$ in $\mathcal R(k+1,k)$ and $\mathcal R(k+1,k+1)$. 
Similar to BURA-method, the numerical computation of $\bfu_r$ involves solving 
of $k+1$ independent linear systems with nonnegative diagonal shifts of $\calAt$. 

\subsection{Q-method}
The solver, proposed by Bonito and Pasciak in \cite{BP15}, incorporates an exponentially convergent quadrature 
scheme for the approximate computation of an integral solution representation, i.e., uses the rational function
\begin{equation*}
Q_\alpha(t):=\frac{2k'\sin(\pi\alpha)}{\pi}\sum_{\ell=-m}^M 
\frac{e^{2(\alpha-1)\ell k'}}{t+e^{-2\ell k'}},\qquad t\in(0,\infty), 
\end{equation*}
where $m=\lceil (1-\alpha)k\rceil$, $M=\lceil \alpha k\rceil$, $k'=\pi/(2\sqrt{\alpha(1-\alpha)k})$. Since 
$$\lceil (1-\alpha)k\rceil+\lceil \alpha k\rceil=\left\{\begin{array}{lcl} k+1, && \alpha k\notin\ZZ\\k, && \alpha k\in\ZZ\end{array}\right.$$
$Q_\alpha$ is either a $(k+1,k+1)$ or a $(k+2,k+2)$ rational function. The approximant of $\bfu_h$ has the form
\begin{equation}\label{eq:Q-method}
\bfu_{Q}:=\frac{2k'\sin(\pi\alpha)}{\pi}\sum_{\ell=-m}^M 
e^{2(\alpha-1)\ell k'}\left(\calAt+e^{-2\ell k'} \calIt \right)^{-1}\bff. 
\end{equation}
The parameter $k'>0$ controls the accuracy of $\bfu_Q$ and the number of linear systems to be solved. 
For example, $k'=1/3$ gives rise to 120 systems for $\alpha=\{0.25,0.75\}$ and $91$ systems for $\alpha=0.5$ 
guaranteeing  $\|\bfu_Q-\bfu \|_2  \approx 
10^{-7}\|\bff\|_2$. We have 
\begin{equation}\label{eq:Q-method error analysis}
\frac{\|\bfu_Q-\bfu \|_2}{\|\bff\|_2}\le\max_{\lambda_i}\left|Q_\alpha(\lambda_i)-\lambda_i^{-\alpha}\right|
\approx \left|Q_\alpha(\lambda_1)-\lambda^{-\alpha}_1\right|.
\end{equation}

Finally, the error analysis, developed in \cite{BP15} states
\begin{equation}\label{eq: Q error}
\lim_{k \to \infty} e^{\pi\sqrt{\alpha(1-\alpha)k}}  \|\bfu_Q-\bfu \|_2 
= \frac{2\sin(\pi\alpha)}{\pi}\left(\frac{1}{\alpha}+
\frac{1}{(1-\alpha)\lambda_1}\right)\|\bff\|_2. 
\end{equation}
\begin{remark}
Varying the quadrature formulae, a family of related methods 
can be obtained. For example, Gauss-Jacobi quadrature rule is 
used to approximate the integral representation of the solution 
in \cite{Aceto_17}.
\end{remark}
\subsection{Comparison of the three solvers}
Comparing \eqref{eq: BURA error} with \eqref{eq: Q error} we observe exponential decay of both errors with respect to the number of linear systems to be solved. The exponential order of the BURA 
estimate is at least twice higher than the one for the quadrature rule, but there is a multiplicative 
factor $\CC^{1-\alpha}$ in \eqref{eq: BURA error}, which depends on the mesh size $h$ and $\CC\to\infty$ as $h\to 0$. 
This implies trade-off between numerical accuracy and computational efficiency for the BURA method. 
The choice of $k$ for $r_\alpha \in \mathcal R (k,k)$ should be synchronized with $h$, while the size of $h$ 
does not affect the choice of $k$ for $Q_\alpha$. Another difference between the two approaches is that the 
error bound in \eqref{eq: BURA error} is unbalanced and can be reached only for $\bff=\bPsi_1$ and only if 
$\CC^{-1}\lambda_1$ is an extreme point for $r_\alpha$ (see \eqref{eq:BURA error analysis}), while the 
error bound in \eqref{eq: Q error} is balanced. Hence, the BURA error heavily depends on the decomposition 
of the right-hand-side $\bff$ along $\{\bPsi_i\}$ and possesses a wide range of values, while the 
quadrature error is  independent on $\bff$. 

The errors in \eqref{eq:BURA error analysis} and \eqref{eq:R-BURA error analysis} are bounded by the expressions 
$\CC^{1-\alpha}E_{\alpha,k,k}/\lambda_1$ and $\CC^\alpha E_{1-\alpha,k,k}/\lambda^{2\alpha}_1$. 
Since $E_{\alpha,k,k}$ is monotonically increasing  function with respect to $\alpha$ and $\CC,\lambda_1>1$, 
for $\alpha>0.5$ the R-BURA method provides better theoretical error bounds, while for $\alpha<0.5$ so 
does the BURA one. 
In the case $\alpha=0.5$ the two approaches behave similarly. The drawback for the R-BURA 
method is the additional condition on $k$ and $h$, namely $\mu_1\in\bigcup\limits_{i=1}^k[\xi_{2i},\xi_{2i+1}]$. 
On the other hand, if we can guarantee this, then the R-BURA method has some advantage, as we solve one linear 
system less ($k+1$ for BURA vs $k$ for R-BURA, when using the same $(k,k)$ function $r_{0.5}(t)$). 
Below we provide an experimental comparison of these approaches for various $h$ and $k=\{7,8,9\}$.  

\section{Numerical tests: comparative analysis and proof of concept}\label{sec:tests}
We consider the fractional Laplace problem with homogeneous boundary conditions \eqref{FracLaplace} in both 1-D and 2-D. 
In 1-D we use the well-known eigenvectors and eigenvalues of the corresponding SPD matrix for experimental validation of the theoretical error analysis.
In 2-D we investigate the relation between numerical accuracy and computational efficiency of the considered three solvers.

\subsection{Algorithm for computing BURA} \label{ss:Remez}

Following \cite{HLMMV18} we consider $\alpha=\{0.25,0.5,0.75\}$ and investigate methods with similar computational 
efficiency.  The rational functions $r_\alpha(t)$ are computed using the modified Remez algorithm, 
described in \cite[Section 3.2]{HLMMV18}. In the case $\alpha=0.25$ we compare $(k,k)$-BURA with the 
$k$-Q-method, $k=9$. The corresponding 
numerical solvers incorporate $10$, respectively $11$ linear systems with positive diagonal shifts of $\calAt$. 
In the cases $\alpha=\{0.5,0.75\}$ we compare $(k,k)$-BURA, 
$(k+1,k)$-R-BURA, $(k+1,k+1)$-R-BURA, and $k$-Q-method for $k=7$. 
This gives rise to $k+1$ linear systems with positive diagonal shifts of $\calA$ for the first two 
methods and $k+2$ linear systems with positive diagonal shifts for the last method. 

\begin{table}[h!]
\caption{ Errors $E_{\alpha,k,m}$ of $r_\alpha(t)$ for $t\in[0,1]$, used in the analysis of 
 BURA and R-BURA numerical methods.}\label{t:error}
\centering
\begin{tabular}{|c|c|c|c|c|c|c|c|c|}
\hline
$\alpha$&$E_{\alpha,5,5}$&$E_{\alpha,6,6}$&$E_{\alpha,7,7}$&
$E_{\alpha,8,7}$& $E_{\alpha,8,8}$ & $E_{\alpha,9,8}$ & $E_{\alpha,9,9}$ & $E_{\alpha,10,9}$\\ \hline
  0.25 & 2.8676e-5 & 9.2522e-6 & 3.2566e-6 & 1.9500e-6 & 1.2288e-6 & 7.5972e-7 & 4.9096e-7 & 3.1128e-7\\ 
  0.50 & 2.6896e-4 & 1.0747e-4 & 4.6037e-5 & 3.0789e-5 & 2.0852e-5 & $\ast$ & $\ast$ & $\ast$
  \\
  0.75 & 2.7348e-3 & 1.4312e-3 & 7.8269e-4 & $\ast$ & $\ast$ & $\ast$ & $\ast$ & $\ast$\\\hline
\end{tabular}
\end{table}

The maximal approximation errors of the involved BURA functions are summarized in 
Table~\ref{t:error}.  We use $\ast$ to indicate errors that cannot be computed when the 
Quadruple-precision floating-point format is applied for the arithmetics. 
The first four zeros of the associated functions 
$\varepsilon(t)=r_{1-\alpha}(t)-t^\alpha$ are presented in Table~\ref{tab:R-BURA roots}. 
Note that they are needed only in the analysis of
R-BURA setting, thus we exclude 
$\alpha=0.25$ where R-BURA behaves worse than BURA. 
%
%


\begin{table}[t]
\centering
\caption{First four roots of $\varepsilon(t)$ for   $r_{1-\alpha} \in \mathcal R(k,m)$.
}\label{tab:R-BURA roots}
 \begin{tabular}{|c||c|c|c|c||c|c|c|c|}
\hline
\multirow{2}{*}{} & \multicolumn{8}{|c|}{First four zeros $\xi_1, \xi_2, \xi_3, \xi_4$ of $\varepsilon(t)=r_{1-\alpha}(t)-t^\alpha$ 
}\\ 
\cline{2-9}
\multirow{2}{*}{$(k,m)$} & \multicolumn{4}{|c||}{$\alpha=0.50$}  & \multicolumn{4}{|c|}{$\alpha=0.75$}  \\ 
\cline{2-9} & $\xi_1$ & $\xi_2$ & $\xi_3$ & $\xi_4$  &  $\xi_1$ & $\xi_2$ & $\xi_3$ & $\xi_4$ \\ 
\hline
\multirow{9}{*}{}  
             $(5,5)$ & 1.030e-7 & 6.732e-6 & 6.592e-5 & 4.352e-4  & 2.185e-6 & 7.269e-5 & 5.004e-4 & 2.353e-3 \\
             $(6,6)$ & 1.650e-8 & 1.076e-6 & 1.053e-5 & 6.950e-5 & 4.836e-7 & 1.609e-5 & 1.108e-4 & 5.216e-4 \\
             $(7,7)$ & 3.100e-9 & 1.981e-7 & 1.932e-6 & 1.275e-5 & 1.202e-7 & 3.999e-6 & 2.754e-5 & 1.297e-4 \\
             $(8,7)$ & 1.400e-9 & 8.840e-8 & 8.644e-7 & 5.705e-6  & 6.070e-8 & 2.019e-6 & 1.390e-5 & 6.544e-5  \\
             $(8,8)$ & 7.00e-10 & 4.070e-8 & 3.967e-7 & 2.617e-6 & 3.280e-8 & 1.091e-6 & 7.509e-6 & 3.536e-5 \\\hline
\end{tabular}
\end{table}

\subsection{Numerical results for 1-D fractional Laplace problem}\label{ss:1D}
For this problem we have $\lambda_1=(4/h^2)\sin^2(\pi h/2)$, $\CC=4/h^2$,  and  $\mu_1=\sin^2(\pi h/2)$. 
The  numerical results are  given on Figures \ref{fig1} -- \ref{fig2}.

First, we solve the system
%
$\calAt^\alpha \bfu=\bPsi_1$. The 
theoretical errors of the three methods, given by 
$\CC^{-\alpha}|\mu^{-1}_1 r_\alpha(\mu_1) - \mu^{-\alpha}_1|$ of BURA,
$\CC^{-\alpha}|r_{1-\alpha}(\mu_1)-\mu_1^{\alpha}|/(\mu^\alpha_1 r_{1-\alpha}(\mu_1))$ of R-BURA, 
and $|Q_\alpha(\lambda_1)-\lambda^{-\alpha}_1|$ of the Q-method,  are presented as function of 
$h\in[10^{-7},10^{-2}]$  on Figure \ref{fig1}.  The numerical results in each graph are obtained by 
comparable computational complexity (expressed through the number of systems solved).

\begin{figure}[h!]
\begin{tabular}{ccc}
\includegraphics[width=0.31\textwidth]{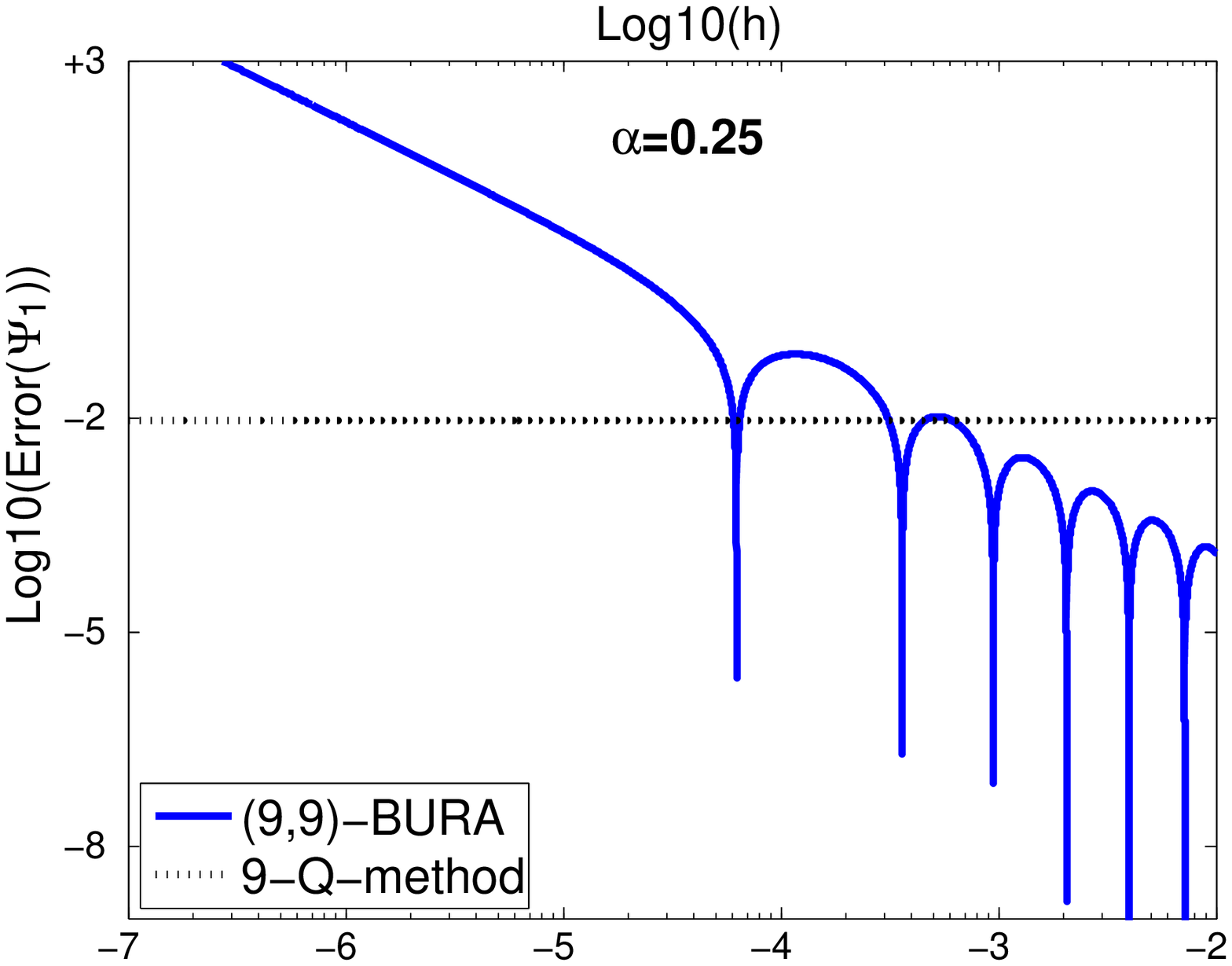} &
\includegraphics[width=0.31\textwidth]{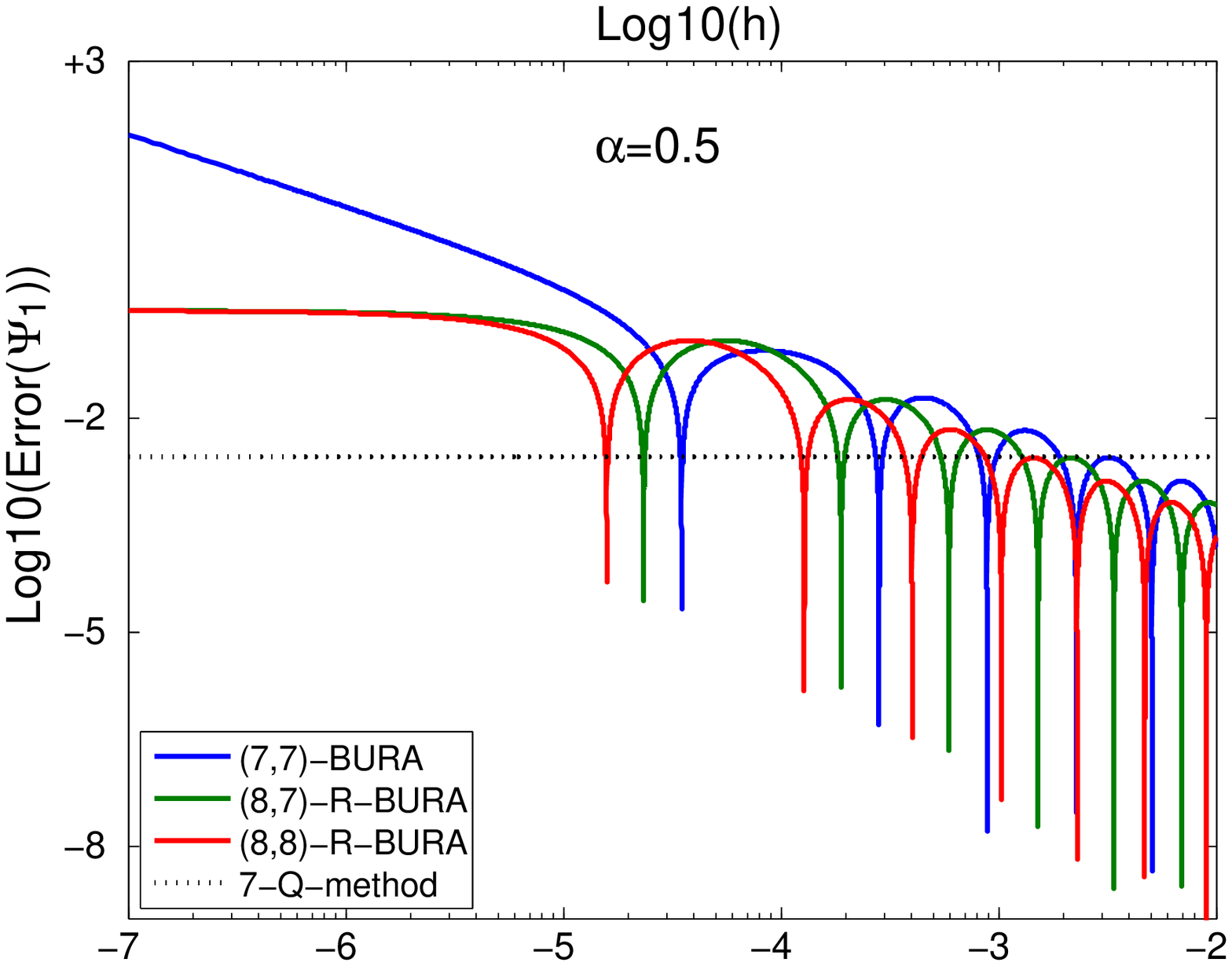} &
\includegraphics[width=0.31\textwidth]{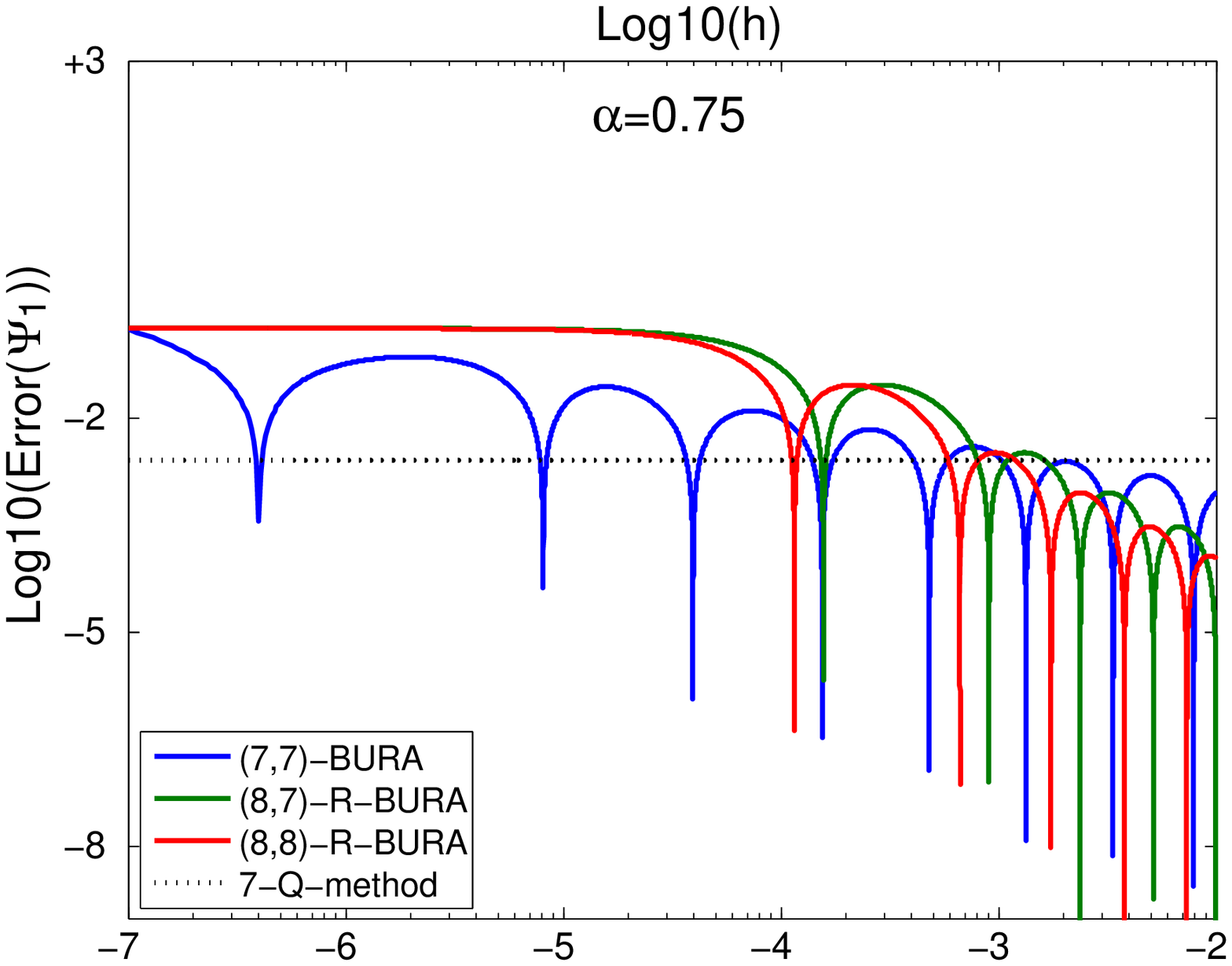} 
\end{tabular}
\caption{Comparison of the theoretical error bounds $\|\bfu_r-\bfu\|_2/\|\bff\|_2$, $\bff=\bPsi_1$, of the three solvers with respect to $h$ for the 1-D fractional 
Laplacian.}\label{fig1}
\end{figure}

The oscillating behavior 
of the BURA-related error is due to the placement of $\mu_1$ with respect to the extreme points of 
$\varepsilon(t)$. When $\alpha=0.75$, $k=8$ (right plot) and $h<10^{-4}$ we observe the constant asymptotic behavior 
of the R-BURA errors towards $\pi^{-2\alpha}$ as $h\to 0$. Similar observation is made for $\alpha=0.5$ 
and $h<2\cdot10^{-5}$. Since $\mu_1\approx \pi^2h^2/4$, we have that $\mu_1\approx2.5\cdot10^{-8}$ for 
$\alpha=0.75$ and $h=10^{-4}$, which, as seen from Table~\ref{tab:R-BURA roots}, is close to the first zero $\xi_1$ 
of $\varepsilon(t)$ for the corresponding $(8,7)-$ and $(8,8)-$ approximations $r_{0.25}(t)$. 
This asymptotic behavior perfectly agrees with the 
error estimate \eqref{eq:R-BURA error analysis at 0}. The same analysis can be made for $\alpha=0.5$ 
and $h=2\cdot10^{-5}$, where $\mu_1\approx5\cdot10^{-10}$.  
The Q-method errors are independent of $h$. We observe that for $\alpha=0.5$ the BURA and R-BURA 
solvers have comparable accuracy over the whole interval $(0,1]$. For $\alpha=0.75$ and $h\in[10^{-4},10^{-3}]$, 
we observe that both $(8,7)-$ and $(8,8)$-R-BURA functions give worse relative errors than the $(7,7)$-BURA function, 
since $\mu_1\in[\xi_1,\xi_2]$ (see Table~\ref{tab:R-BURA roots}).

\begin{figure}[h!]
\begin{tabular}{ccc}
\includegraphics[width=0.31\textwidth]{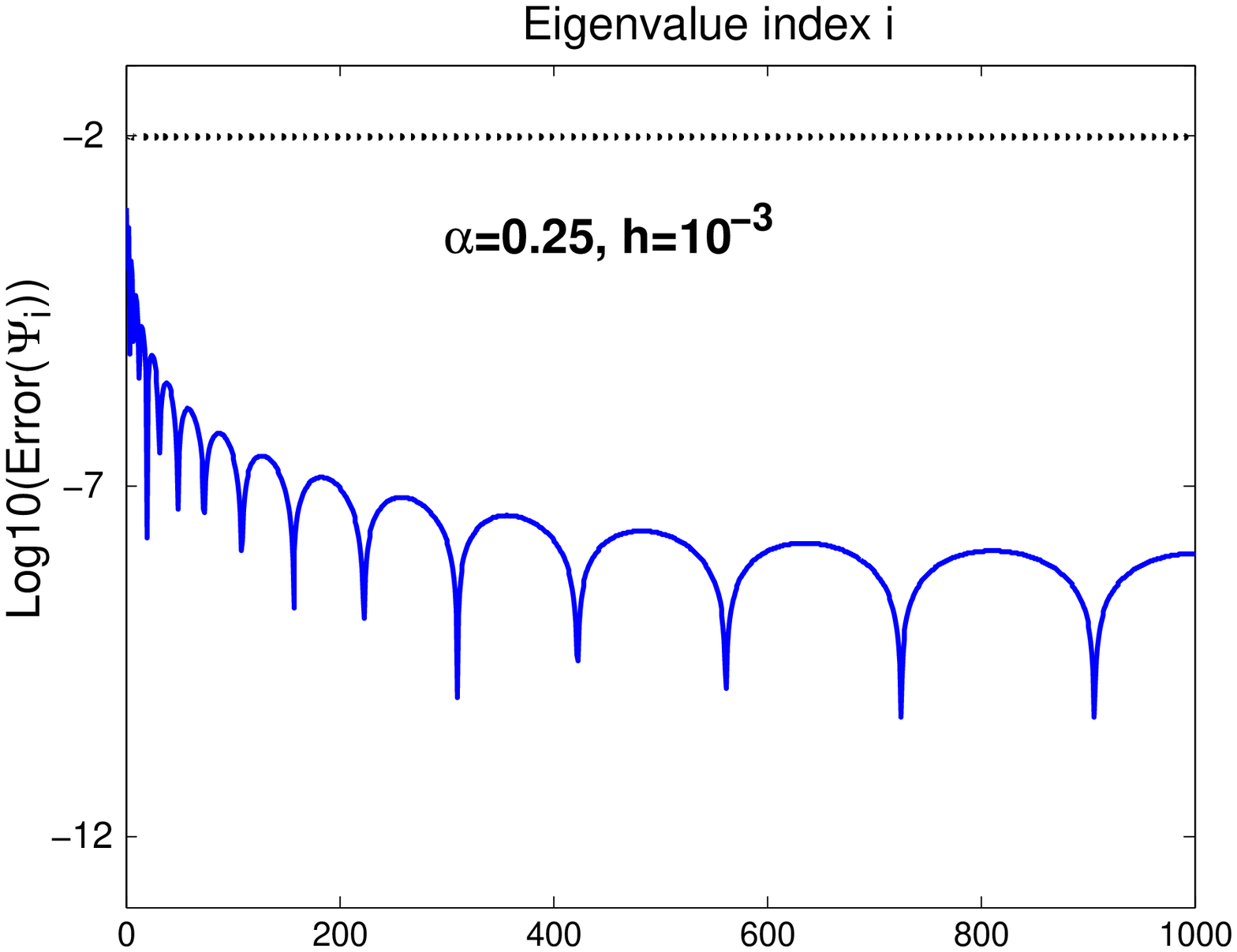} &
\includegraphics[width=0.31\textwidth]{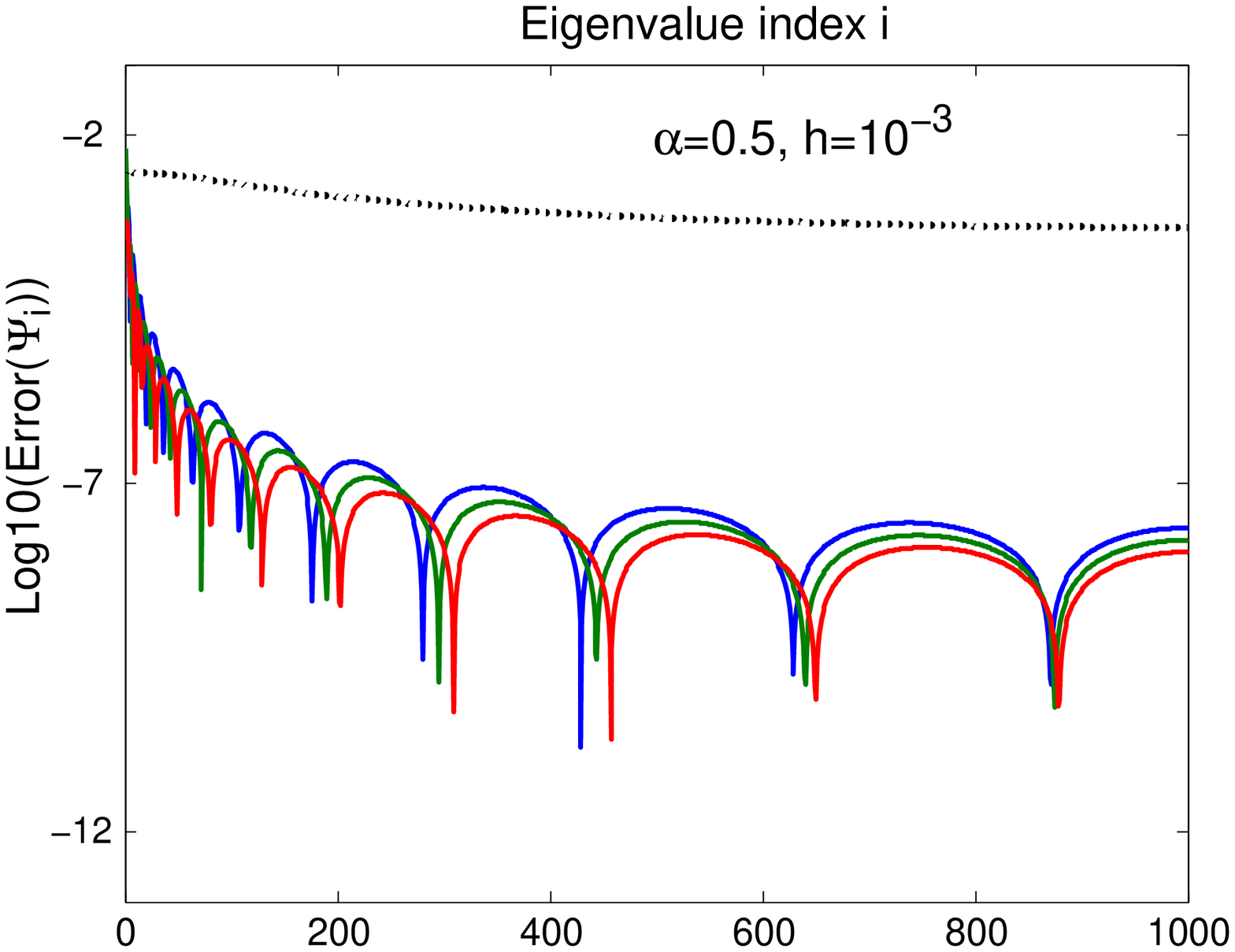} &
\includegraphics[width=0.31\textwidth]{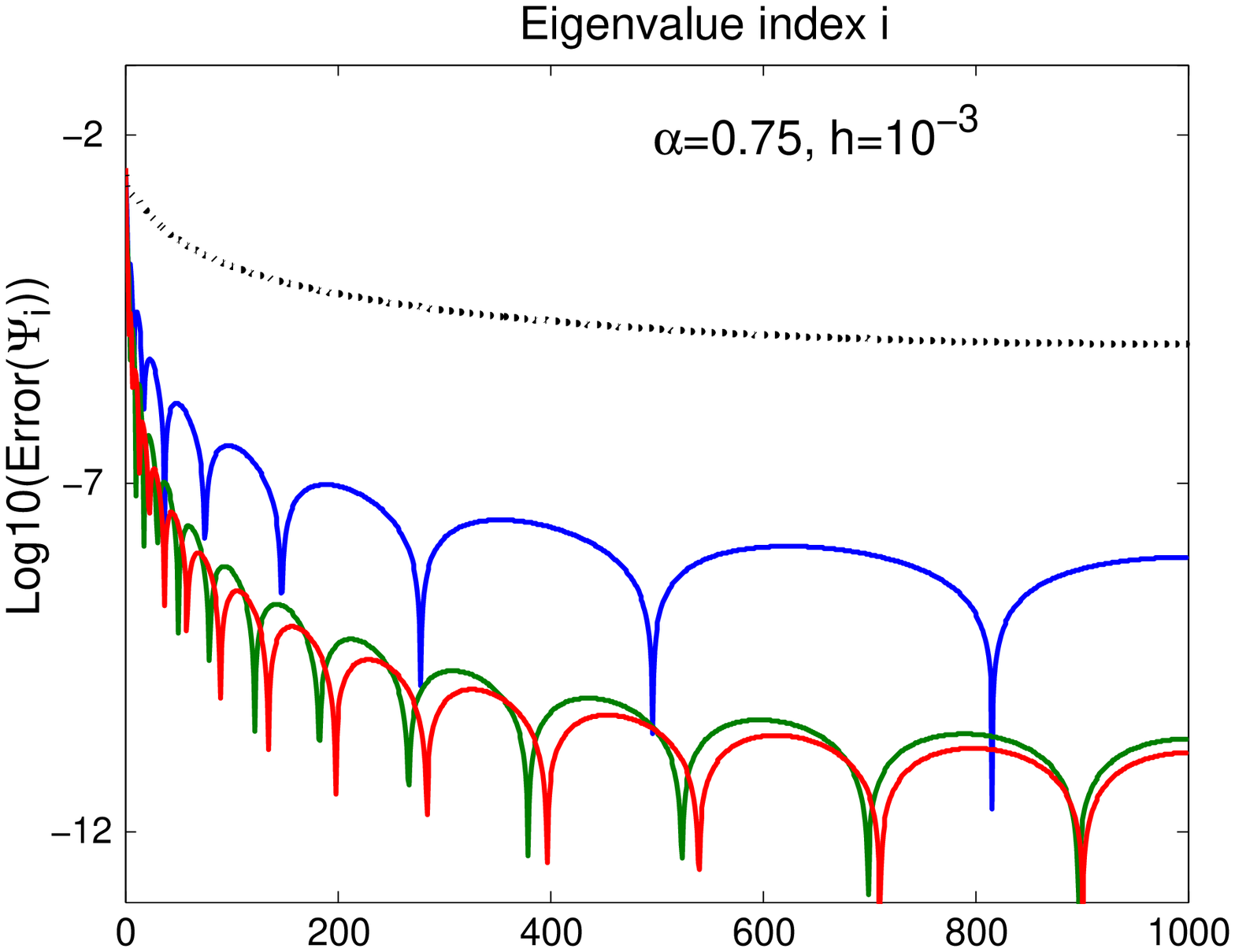}\\
\includegraphics[width=0.31\textwidth]{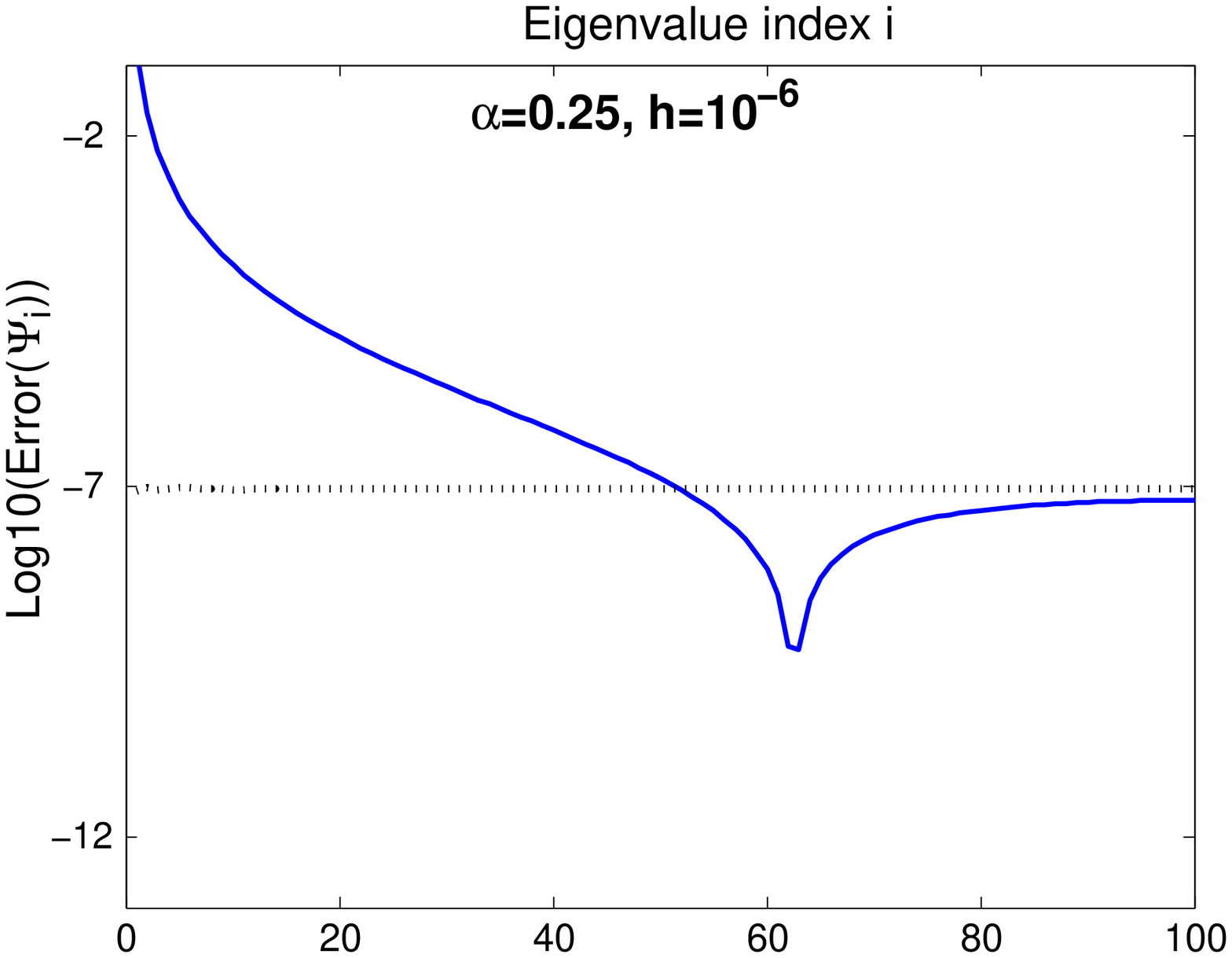} &
\includegraphics[width=0.31\textwidth]{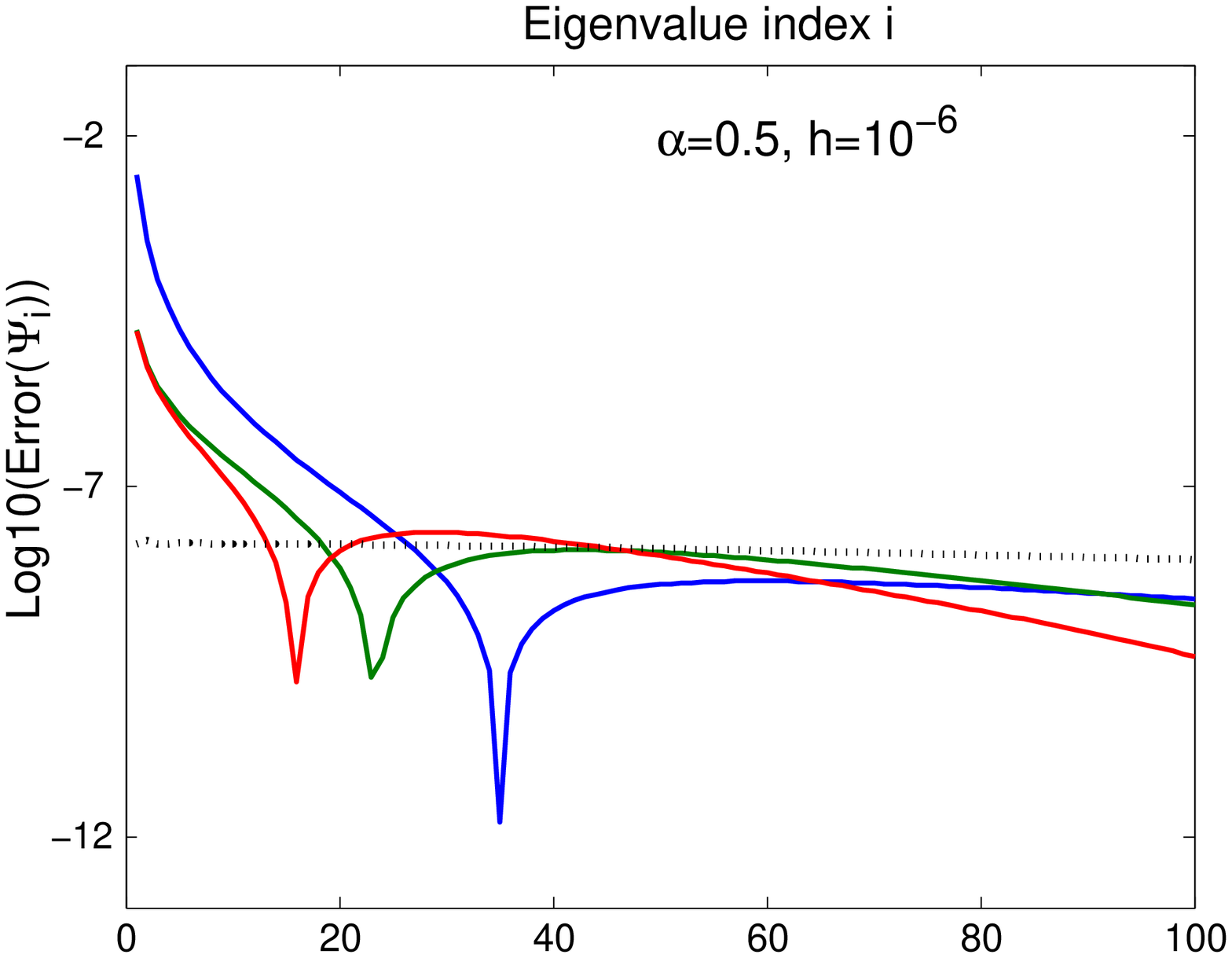} &
\includegraphics[width=0.31\textwidth]{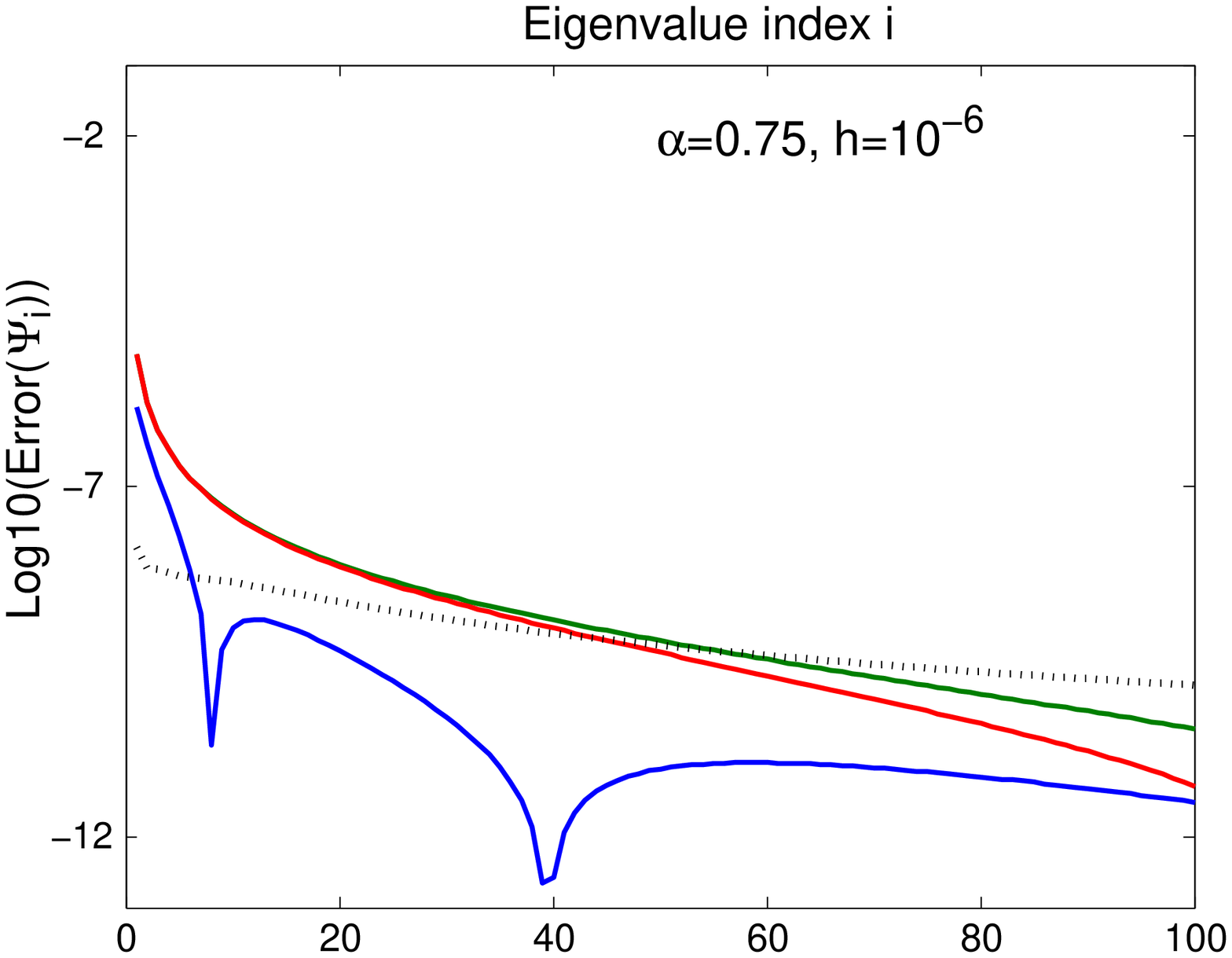}
\end{tabular}
\caption{Spectral decomposition of the error for the 1-D fractional Laplace problem with $h=10^{-3}$ (top) and 
$h=10^{-6}$ (bottom). Colors are with respect to the legend of Fig.~\ref{fig1}. Top: for each $i$ we plot the 
corresponding errors $\|\bfu_r-\bfu_Q\|_2/\|\bff\|_2$ for $\bff=\bPsi_i$.  The bottom plots present the corresponding errors for 
the first $100$ eigenvectors, i.e., for $\{\bPsi_i\}_1^{100}$ .}
\label{fig2}
\end{figure}

The second set of experiments deals with the error over the whole spectrum of $\calAt$ and is presented on 
Figure~\ref{fig2}. For $h=10^{-3}$ and $h=10^{-6}$ we compute 
$\CC^{-\alpha}|\mu^{-1}_i r_\alpha(\mu_i) - \mu^{-\alpha}_i|$, 
$\CC^{-\alpha}|r_{1-\alpha}(\mu_i)-\mu_i^{\alpha}|/(\mu^\alpha_i r_{1-\alpha}(\mu_i))$, 
and $|Q_\alpha(\lambda_i)-\lambda^{-\alpha}_i|$ for all $i$, which is equivalent to letting the 
right-hand-side in \eqref{FD4Laplace} run over the eigenvectors of $\calAt$ ($\bff=\bPsi_i$). 
The plots on the first row illustrate the complete spectral decomposition of the error for $h=10^{-3}$. 
For $h=10^{-6}$ we show the spectral error over the first $1\%$ of the eigen-modes on the second row. 
The unbalanced behavior of the BURA-related errors in contrast to the 
balanced behavior of the errors of the Q-method is clearly observed. 
High-frequency modes are practically perfectly 
reconstructed by the R-BURA methods, while the low-frequency ones lead to larger errors. 
When $h=10^{-3}$ and $k$ is chosen 
accordingly, all BURA-related errors are smaller than the corresponding Q-related errors.
When $h=10^{-6}$ and $k$ is chosen poorly, then the BURA and R-BURA errors on the first 
several eigenvectors can be significantly larger than the corresponding Q-method errors.
However, among a million of eigenvectors, the Q-method outperforms the BURA methods on not more than $50$ of them.  Comparing BURA to R-BURA approaches, we experimentally confirm that 
the two methods behave similarly when $\alpha=0.5$, while R-BURA is better for $\alpha=0.75$.

\subsection{2-D numerical experiments}\label{ss:2D}
We consider the finite difference approximation of \eqref{FracLaplace} with two different r.h.s., namely, $f_1$ and $f_2$: 
\begin{equation}\label{eq: rhs}
f_1(x,y)=\left\{\begin{array}{rl} 1,& \text{if } (x-0.5)(y-0.5)>0,\\ -1,& \text{otherwise}.\end{array}\right.\quad
f_2(x,y)=\cos(\pi hx)\cos(\pi hy)
\end{equation}
The function $f_1$ has a jump discontinuity along $x=0.5$ and $y=0.5$ and has already been used as a test function in this framework \cite{BP15,HLMMV18}. In this case $\lambda_1 
\approx2\pi^2$, 
$\CC=\|\calAt\|_\infty=8h^{-2}$, and $\mu_1=\sin^2(\pi h/2)$.
%
The reference solution $\bfu_Q$ is generated by the Q-method with $k^\prime=1/3$ on a fine mesh with 
mesh-size $h=2^{-12}$. 
Note that $\bfu_Q$ is an approximation to the exact solution $\bfu$ with six correct digits, 
$\|\bfu_Q-\bfu\|_2/\|\bff\|_2 \approx 10^{-7}$. 
The numerical results are summarized in Tables~\ref{tab:2D results 1}--\ref{tab:2D results 3}. 
The presented relative $\ell_2$-errors illustrate the theoretical analysis, while $\ell_\infty$-errors are given as 
additional information.
 

\begin{table}[h!]
\centering
\caption{Relative errors for various discretization levels and $\alpha=0.25$. 
}\label{tab:2D results 1}
 \begin{tabular}{|c|c|cc|cc|cc|cc|}
\hline
\multirow{3}{*}{\hspace{1ex}$k$\hspace{1ex}} & \multirow{3}{*}{$h$} & \multicolumn{4}{|c|}{Checkerboard rhs} & \multicolumn{4}{|c|}{Tensor product cosine rhs}\\ \cline{3-10}
& & \multicolumn{2}{|c|}{$(k,k)$-BURA} & \multicolumn{2}{|c|}{$k$-Q-method} & \multicolumn{2}{|c|}{$(k,k)$-BURA} & \multicolumn{2}{|c|}{$k$-Q-method} \\ 
& &  $\ell_2$ & $\ell_\infty$ &  $\ell_2$ & $\ell_\infty$ &  $\ell_2$ & $\ell_\infty$ &  $\ell_2$ & $\ell_\infty$ \\ \hline
\multirow{5}{*}{$9$} & 
  $2^{-8}$  & 5.863e-3 & 5.236e-2 & 1.080e-2 & 4.285e-2 & 2.781e-4 & 2.600e-3 & 6.823e-3 & 9.381e-3 \\
& $2^{-9}$  & 2.823e-3 & 3.234e-2 & 9.707e-3 & 2.425e-2 & 1.441e-4 & 1.813e-3 & 6.752e-3 & 8.586e-3 \\
& $2^{-10}$ & 1.253e-3 & 1.785e-2 & 9.436e-3 & 1.870e-2 & 2.268e-4 & 1.210e-3 & 6.726e-3 & 7.984e-3 \\
& $2^{-11}$ & 5.027e-4 & 7.443e-3 & 9.381e-3 & 1.383e-2 & 4.888e-4 & 7.707e-4 & 6.717e-3 & 7.412e-3 \\
& $2^{-12}$ & 4.883e-3 & 1.019e-2 & 9.374e-3 & 9.568e-3 & 4.425e-3 & 5.031e-3 & 6.713e-3 & 6.790e-3 \\\hline   
\end{tabular}
\end{table} 
The balanced error distribution along the full spectrum for the Q-method gives rise to stable relative errors, 
independent of $h$ for all $\alpha$ on both examples. 
The error distribution for BURA and R-BURA methods depends on $h$ and $k$. 
From Table~\ref{tab:2D results 1}  we see that for $\alpha=0.25$ and $h \ge 2^{-12}$ the choice 
$k=9$ for the $(k,k)$-BURA method has lower $\ell_2$-error 
than the error of the Q-solver for comparable computational work.  

{\small
\begin{table}[t!]
\centering
\caption{Relative errors for $\bff_1$ on various discretization levels and $\alpha=\{0.5,0.75\}$. 
}\label{tab:2D results 2}
 \begin{tabular}{|c|c|cc|cc|cc|cc|}
\hline
\multirow{3}{*}{$(\alpha,k)$} & \multirow{3}{*}{$h$} & \multicolumn{8}{|c|}{Checkerboard right-hand-side}\\ \cline{3-10}
& & \multicolumn{2}{|c|}{$(k,k)$-BURA} & \multicolumn{2}{|c|}{$(k+1,k)$-R-BURA} & \multicolumn{2}{|c|}{$(k+1,k+1)$-R-BURA} & \multicolumn{2}{|c|}{$k$-Q-method} \\ 
& &  $\ell_2$ & $\ell_\infty$ &  $\ell_2$ & $\ell_\infty$ &  $\ell_2$ & $\ell_\infty$ &  $\ell_2$ & $\ell_\infty$ \\ \hline
\multirow{5}{*}{$(0.50,7)$} & 
  $2^{-8}$  & 1.383e-3 & 6.814e-3 & 1.351e-3 & 6.820e-3 & 1.347e-3 & 6.806e-3 & 3.113e-3 & 6.800e-3 \\
& $2^{-9}$  & 8.692e-4 & 3.503e-3 & 6.777e-4 & 3.497e-3 & 6.687e-4 & 3.497e-3 & 2.895e-3 & 4.573e-3 \\
& $2^{-10}$ & 7.657e-4 & 1.808e-3 & 7.845e-4 & 1.766e-3 & 4.660e-4 & 1.619e-3 & 2.841e-3 & 3.552e-3 \\
& $2^{-11}$ & 8.243e-4 & 1.447e-3 & 1.879e-3 & 4.204e-3 & 2.583e-4 & 6.293e-4 & 2.830e-3 & 3.078e-3 \\
& $2^{-12}$ & 5.423e-3 & 1.135e-2 & 1.976e-3 & 4.831e-3 & 1.447e-3 & 2.861e-3 & 2.828e-3 & 2.902e-3 \\\hline  
\multirow{5}{*}{$(0.75,7)$} & 
  $2^{-8}$  & 4.194e-4 & 1.277e-3 & 4.226e-4 & 1.278e-3 & 4.206e-4 & 1.272e-3 & 1.558e-3 & 2.276e-3 \\
& $2^{-9}$  & 2.509e-4 & 6.038e-4 & 2.281e-4 & 6.053e-4 & 1.967e-4 & 6.029e-4 & 1.514e-3 & 1.984e-3 \\
& $2^{-10}$ & 4.264e-4 & 9.644e-4 & 2.410e-4 & 5.451e-4 & 1.234e-4 & 2.604e-4 & 1.503e-3 & 1.887e-3 \\
& $2^{-11}$ & 5.222e-4 & 1.206e-3 & 2.128e-4 & 3.923e-4 & 5.601e-4 & 1.185e-3 & 1.500e-3 & 1.843e-3 \\
& $2^{-12}$ & 6.560e-5 & 1.420e-4 & 3.077e-3 & 6.268e-3 & 1.316e-3 & 2.819e-3 & 1.499e-3 & 1.823e-3 \\\hline 
\end{tabular}
\end{table}
}

Next set of numerical experiments is presented in Tables~\ref{tab:2D results 2} -- \ref{tab:2D results 3}. 
For $\alpha=0.5$ and $k=7$, we have $\mu_1>\xi_2$ for $h\ge2^{-12}$ and both $(8,7)$-  and $(8,8)$-R-BURA 
methods are reliable. Their $\ell_2$ 
relative errors are smaller than the corresponding errors for the $k$-Q-solver on all considered mesh sizes. 
The $(8,8)$-R-BURA solver is more accurate than the $(8,7)$-R-BURA one. The choice $k=7$ for 
the BURA solver is reliable for $h\ge2^{-11}$, but for $h=2^{-12}$ 
a larger k is needed. Like the 1-D case, BURA and R-BURA 
solvers behave similarly when $k$ is properly chosen.  
{\small
\begin{table}[t!]
\centering
\caption{Relative errors for $\bff_2$ on various discretization levels and $\alpha=\{0.5,0.75\}$.}\label{tab:2D results 3}
 \begin{tabular}{|c|c|cc|cc|cc|cc|}
\hline
\multirow{3}{*}{$(\alpha,k)$} & \multirow{3}{*}{$h$} & \multicolumn{8}{|c|}{Tensor product cosine right-hand-side}\\ \cline{3-10}
& & \multicolumn{2}{|c|}{$(k,k)$-BURA} & \multicolumn{2}{|c|}{$(k+1,k)$-R-BURA} & \multicolumn{2}{|c|}{$(k+1,k+1)$-R-BURA} & \multicolumn{2}{|c|}{$k$-Q-method} \\ 
& &  $\ell_2$ & $\ell_\infty$ &  $\ell_2$ & $\ell_\infty$ &  $\ell_2$ & $\ell_\infty$ &  $\ell_2$ & $\ell_\infty$ \\ \hline
\multirow{5}{*}{$(0.50,7)$} & 
  $2^{-8}$  & 1.509e-4 & 3.299e-4 & 5.790e-5 & 1.810e-4 & 5.031e-5 & 1.809e-4 & 1.423e-3 & 2.901e-3\\
& $2^{-9}$  & 2.717e-4 & 6.065e-4 & 1.179e-4 & 2.901e-4 & 1.075e-4 & 2.438e-4 & 1.418e-3 & 2.867e-3\\
& $2^{-10}$ & 3.432e-4 & 8.118e-4 & 3.292e-4 & 7.450e-4 & 1.801e-4 & 4.192e-4 & 1.415e-3 & 2.858e-3\\
& $2^{-11}$ & 4.545e-4 & 1.282e-3 & 8.335e-4 & 1.817e-3 & 1.648e-4 & 4.218e-4 & 1.415e-3 & 2.856e-3\\
& $2^{-12}$ & 2.420e-3 & 5.190e-3 & 9.188e-4 & 2.139e-3 & 6.791e-4 & 1.635e-3 & 1.414e-3 & 2.856e-3\\\hline
\multirow{5}{*}{$(0.75,7)$} & 
  $2^{-8}$  & 2.222e-5 & 5.484e-5 & 1.893e-5 & 3.896e-5 & 3.586e-6 & 8.906e-6 & 7.386e-4 & 1.422e-3\\
& $2^{-9}$  & 7.383e-5 & 1.836e-4 & 5.171e-5 & 1.086e-4 & 6.334e-6 & 1.750e-5 & 7.367e-4 & 1.420e-3\\
& $2^{-10}$ & 1.861e-4 & 4.059e-4 & 1.024e-4 & 2.215e-4 & 4.212e-5 & 9.762e-5 & 7.358e-4 & 1.420e-3\\
& $2^{-11}$ & 2.333e-4 & 5.147e-4 & 1.125e-4 & 2.941e-4 & 2.492e-4 & 5.215e-4 & 7.354e-4 & 1.420e-3\\
& $2^{-12}$ & 9.046e-5 & 2.771e-4 & 1.369e-3 & 2.802e-3 & 5.864e-4 & 1.238e-3 & 7.353e-4 & 1.420e-3\\\hline
\end{tabular}
\end{table}

For $\alpha=0.75$ and $h<2^{-10}$, we have $\mu_1\in[\xi_1,\xi_2]$ for both $(8,7)$- and $(8,8)$-R-BURA methods. 
As a result the $(8,8)$-R-BURA solution is less accurate than the one obtained by $(7,7)$-BURA  for 
$h=\{2^{-11},2^{-12}\}$, while the $(8,7)$-R-BURA solver is outperformed by all the other three for 
$h=2^{-12}$. Once we guarantee that $\mu_1>\xi_2$, the R-BURA approach gives 
rise to the highest accuracy. Again, the $(8,8)$-R-BURA solution is more accurate than the one
obtained by $(8,7)$-R-BURA.

Finally, we present a comparison on numerical accuracy versus computational efficiency for properly chosen 
$k$ and $h$. We fix $h=2^{-10}$ and for each BURA-related method we compute the smallest $k$, such that the 
corresponding $k$-Q-method gives smaller  relative $\ell_2$-error for $\bff_1$ and $\bff_2$. When $\alpha=0.25$ 
the $(9,9)$-BURA solver has lower accuracy  than the $37$-Q-method for $\bff_1$ and the $36$-Q-method for 
$\bff_2$, respectively. This means that, instead of the 10 linear systems incorporated in the BURA-method, we 
need to solve 39, respectively 37, linear systems for the Q-method. For $\alpha=0.5$ and both $\{\bff_1,\bff_2\}$ 
we need to use $k=16$ for the Q-method to get better accuracy than $(7,7)$-BURA, $k=16$ for the Q-method to 
get better accuracy than $(8,7)$-R-BURA, and $k=20$ for the Q-method to get better accuracy than $(8,8)$-R-BURA. 
For $\bff_1$ and $\alpha=0.75$ we need to use $k=13$ for the Q-method to get better accuracy than $(7,7)$-BURA, 
$k=17$ for the Q-method to get better accuracy than $(8,7)$-R-BURA, and $k=25$ for the Q-method to get better 
accuracy than $(8,8)$-R-BURA. Finally, for $\bff_2$ and $\alpha=0.75$ we need to use $k=13$ for the Q-method 
to get better accuracy than $(7,7)$-BURA, $k=21$ for the Q-method to get better accuracy than $(8,7)$-R-BURA, 
and $k=29$ for the Q-method to get better accuracy than $(8,8)$-R-BURA. Therefore, with respect to numerical 
accuracy versus computational efficiency the R-BURA solver for $\alpha=0.75$ behaves similarly to the BURA 
solver for $\alpha=0.25$ and can be up to four times more efficient than the corresponding Q-solver.    

\section{Concluding remarks}\label{sec:conclusion}
We present a comparative analysis of three methods for solving
equations involving fractional powers of elliptic operators, namely, the method of Bonito and Pasciak, \cite{BP15},
BURA method based on the best rational approximation of $t^{1-\alpha}$, \cite{HLMMV18}, and the new
method, R-BURA, based on the best rational approximation of $t^{\alpha}$ on $[0,1]$.

The method of Bonito and Pasciak, \cite{BP15},
uses Sinc quadratures and has exponential convergence 
with respect to the number of quadrature nodes. The BURA method, \cite{HLMMV18}, 
has exponential convergence as well, is accurate for 
$\alpha$ close to $0$, and performs well for fixed step-size $h$. 
The new method, R-BURA, has also exponential convergence with respect to the degree of the rational approximation for fixed step-size $h$. In contrast to BURA, R-BURA method performs better for $\alpha$ close to $1$. However, the accuracy of both methods deteriorates when $h \to 0$.

Based on this study, we expect that 
one could be able to construct a method that combines the
advantages of these approaches, computational efficiency and exponential convergence rate.

\section*{Acknowledgement}
This research has been partially supported by the Bulgarian National Science Fund under grant No. BNSF-DN12/1. The work of  R. Lazarov has been partially supported by the grant NSF-DMS \#1620318. The work of S. Harizanov has been partially supported by the Bulgarian National Science Fund under grant No. BNSF-DM02/2.

\bibliographystyle{abbrv}
\bibliography{Strobl_Arxiv_2018}

\begin{thebibliography}{10}

\bibitem{Aceto_17}
L.~Aceto and P.~Novati.
\newblock Rational approximation to the fractional {L}aplacian operator in
  reaction-diffusion problems.
\newblock {\em SIAM J. Sci. Comput.}, 39(1):A214 -- A228, 2017.

\bibitem{BP15}
A.~Bonito and J.~Pasciak.
\newblock Numerical approximation of fractional powers of elliptic operators.
\newblock {\em Mathematics of Computation}, 84(295):2083--2110, 2015.

\bibitem{druskin1998extended}
V.~Druskin and L.~Knizhnerman.
\newblock Extended {K}rylov subspaces: approximation of the matrix square root
  and related functions.
\newblock {\em SIAM Journal on Matrix Analysis and Applications},
  19(3):755--771, 1998.

\bibitem{SFilip18}
S.-I. Filip, Y.~Nakatsukasa, L.~N. Trefethen, and B.~Beckermann.
\newblock Rational minimax approximation via adaptive barycentric
  representations.
\newblock {\em arXiv preprint arXiv:1705.10132v2}, 2018.

\bibitem{HLMMV18}
S.~Harizanov, R.~Lazarov, S.~Margenov, P.~Marinov, and Y.~Vutov.
\newblock Optimal solvers for linear systems with fractional powers of sparse
  {SPD} matrices.
\newblock {\em Numerical Linear Algebra with Applications}, 25(4):115--128,
  2018.

\bibitem{harizanov2017positive}
S.~Harizanov and S.~Margenov.
\newblock Positive approximations of the inverse of fractional powers of {SPD
  M}-matrices.
\newblock {\em arXiv preprint arXiv:1706.07620}, 2017.

\bibitem{Higham1997}
N.~J. Higham.
\newblock Stable iterations for the matrix square root.
\newblock {\em Numerical Algorithms}, 15(2):227--242, 1997.

\bibitem{ilic2005numerical}
M.~Ili{\'c}, F.~Liu, I.~W. Turner, and V.~Anh.
\newblock Numerical approximation of a fractional-in-space diffusion equation,
  {I}.
\newblock {\em Fractional Calculus and Applied Analysis}, 8(3):323--341, 2005.

\bibitem{ilic2009numerical}
M.~Ili{\'c}, I.~W. Turner, and V.~Anh.
\newblock A numerical solution using an adaptively preconditioned {L}anczos
  method for a class of linear systems related with the fractional {P}oisson
  equation.
\newblock {\em International Journal of Stochastic Analysis}, 2008, 2009.

\bibitem{Kenney1991}
C.~Kenney and A.~J. Laub.
\newblock Rational iterative methods for the matrix sign function.
\newblock {\em SIAM J. Matrix Anal. Appl.}, 12(2):273--291, Mar. 1991.

\bibitem{Meinardus67}
G.~Meinardus.
\newblock {\em Approximation of {F}unctions: {T}heory and {N}umerical
  {M}ethods}.
\newblock Springer, New York, 1967.

\bibitem{SS93}
E.~B. Saff and H.~Stahl.
\newblock Asymptotic distribution of poles and zeros of best rational
  approximants to $x^\alpha$ on $[0,1]$.
\newblock In {\em "Topics in Complex Analysis", Banach Center Publications},
  volume~31. Institute of Mathematics, Polish Academy of Sciences, Warsaw,
  1995.

\bibitem{Stahl93}
H.~Stahl.
\newblock Best uniform rational approximation of $x^\alpha$ on $[0, 1]$.
\newblock {\em Bulletin of the American Mathematical Society}, 28(1):116--122,
  1993.

\bibitem{stahl2003}
H.~R. Stahl.
\newblock Best uniform rational approximation of $x^\alpha$ on $[0, 1]$.
\newblock {\em Acta Mathematica}, 190(2):241--306, 2003.

\bibitem{Thomee2006}
V.~Thom\'ee.
\newblock {\em Galerkin finite element methods for parabolic problems},
  volume~25 of {\em Springer Series in Computational Mathematics}.
\newblock Springer-Verlag, Berlin, second edition, 2006.

\bibitem{varga1992some}
R.~S. Varga and A.~J. Carpenter.
\newblock Some numerical results on best uniform rational approximation of
  $x^\alpha$ on [0, 1].
\newblock {\em Numerical Algorithms}, 2(2):171--185, 1992.

\end{thebibliography}

\end{document}